\documentclass[a4paper]{amsart}

\usepackage{latexsym}
\usepackage{amssymb}
\usepackage{amsmath}
\usepackage{amsthm}
\usepackage{amsfonts}
\usepackage{color}
\usepackage{enumitem}

\usepackage{pictexwd,dcpic}

\usepackage{graphicx}
\usepackage{subfigure}

\usepackage{psfrag}
\usepackage{hyperref}
\usepackage{comment}

\newtheorem{thm}{Theorem}[section]
\newtheorem{cor}[thm]{Corollary}
\newtheorem{lem}[thm]{Lemma}
\newtheorem{prop}[thm]{Proposition}

\newtheorem*{theorem*}{Main Theorem}

\newtheorem{corollary}{Corollary}

\theoremstyle{remark}
\newtheorem{defn}[thm]{Definition}

\newtheorem{rmk}[thm]{Remark}

\numberwithin{equation}{section}

\newcommand{\scal}[1]{\langle #1 \rangle}

\newcommand{\RR}{\mathbb{R}}

\newcommand{\HH}{\mathbb{H}}
\newcommand{\CC}{\mathbb{C}}

\newcommand{\DD}{\mathbb{D}}

\newcommand{\fol}{\mathcal{F}}

\DeclareMathOperator{\Iso}{Isom}

\newcommand{\SO}{\mathrm{SO}}
\newcommand{\OO}{\mathrm{O}}
\newcommand{\U}{\mathrm{U}}
\newcommand{\SU}{\mathrm{SU}}
\newcommand{\Sp}{\mathrm{Sp}}
\newcommand{\Spin}{\mathrm{Spin}}

\newcommand{\Sym}{\mathrm{Sym}}

\newcommand{\tr}{\mathrm{tr}}

\newcommand{\kk}{\mathfrak{k}}

\newcommand{\Ll}{\mathfrak{l}}

\def\ol{\overline}

\newcommand{\lra}{\longrightarrow}
\newcommand{\lmt}{\longmapsto}

\DeclareMathOperator{\diam}{diam}

\newcommand{\sphere}{\mathrm{\mathbb{S}}}

\begin{document}



\title{On homogeneous composed Clifford foliations}


\author{Claudio Gorodski}
\author{Marco Radeschi}

\date{\today}


\subjclass[2000]{}


\begin{abstract}
We complete the classification, initiated by the second named author, 
of homogeneous singular Riemannian foliations of spheres 
that are lifts of foliations produced from Clifford systems. 
 \end{abstract}

\maketitle




A singular Riemannian foliation of a Riemannian manifold $M$ is, roughly speaking, a partition $\fol$ of $M$ into connected complete submanifolds, not necessarily of the same dimension, that locally stay at a constant distance one from 
another. Singular Riemannian foliations of round spheres $(\sphere^n,\fol)$ are of special importance since, other than producing submanifolds with interesting geometrical properties, they provide local models around a point 
of general singular Riemannian foliations.


The special case of singular Riemannian foliations in spheres 
whose leaves of maximal dimension have codimension one 
is better known as the case of \emph{isoparametric foliations}, and its 
study dates back to \'E. Cartan, who showed the existence of a number 
of non-trivial examples. However, 
his examples were all \emph{homogeneous}, i.e., given as orbits of isometric 
group actions on $\sphere^n$. The first inhomogeneous examples were found 
much later by Ozeki and Takeuchi~\cite{OT1}. A while later, Ferus, 
Karcher and M\"unzner~\cite{FKM} developed an algebraic framework based
on Clifford algebras (or, equivalently, Clifford systems, see 
subsection~\ref{clif-sys}) to construct a large family of examples of 
isoparametric foliations (so called of \emph{FKM-type}), including many 
inhomogeneous examples, and completely classified the homogeneous ones.  

Whereas the theory and classification of isoparametric foliations of spheres
are by now rather well understood, the situation of singular Riemannian 
foliations in higher codimensions is still largely 
\emph{terra incognita}. 
In~\cite{Rad14}, inspired by the ideas in~\cite{FKM}, two new classes of foliations were introduced. Namely, the class of \emph{Clifford foliations}, 
and the class of 
\emph{composed foliations} which properly contains the first one. 
A Clifford foliation $(\sphere^n,\fol_C)$ is constructed from a Clifford system $C$,
and a composed foliation $(\sphere^n,\fol_0\circ\fol_C)$ is constructed from $C$ and a 
singular Riemannian foliation $(\sphere^m,\fol_0)$ 
of a lower dimensional sphere. 
The natural question of determining which ones
are homogeneous was also solved in~\cite{Rad14}, with the exception
of composed foliations based on Clifford systems of type~$C_{8,1}$
and~$C_{9,1}$ (see subsection~\ref{clif-sys}). The goal of the present 
work is to deal with these two remaining, more involved cases.   

\begin{theorem*}\label{T:classification}
Let $(\sphere^{n},\fol_0\circ\fol_C)$ be a homogeneous composed foliation, with either $C=C_{8,1}$ or $C=C_{9,1}$.
\begin{itemize}
\item If $C=C_{8,1}$, then $n=15$, and there are exactly six examples of
 homogeneous foliations $(\sphere^{15},\fol_0\circ\fol_{C_{8,1}})$, listed in Tables~\ref{Table 1} and~\ref{Table 4}.
\item If $C=C_{9,1}$, then $n=31$, and the only homogeneous foliation $(\sphere^{31}, \fol_0\circ \fol_{C_{9,1}})$ is the isoparametric one induced by 
the action of $\Spin(10)$ on $\sphere^{31}$ via the spin 
representation. In this case, $m=9$ and the corresponding foliation 
$(\sphere^9,\fol_0)$ consists of one leaf.
\end{itemize}
\end{theorem*}

There is a general idea that it should be possible to recover many geometric properties 
of the singular Riemannian foliations from the geometry of the 
underlying leaf (or quotient) space, 
compare e.g.~\cite{Lyt,L-T,Wiesen1,G-L3,G-L,G-L2,A-R}. In this regard,
it was shown in~\cite{Rad14} that Clifford foliations are characterized
as those singular Riemannian foliations of spheres whose leaf spaces
isometric to either a sphere or a hemi-sphere of constant curvature~$4$. 
More generally, it was believed that any foliation whose leaf space 
has constant curvature~$4$ should be a composed foliation. Our 
result shows that this belief is now dismissed. 
Comparing our Main Theorem with~\cite[Table~II]{St} and \cite[Table~1]{G-L3}, 
we observe that there are exactly two homogeneous 
foliations on $\sphere^{31}$ whose quotient space has constant sectional 
curvature~$4$ and which are not composed, namely, 
those given by the orbits of $\Spin(9)$ and $\Spin(9)\cdot\SO(2)$ 
actions on $\sphere^{31}$ with 
quotient a quarter of a round sphere~$\frac12\sphere^3_{++}$ and 
an eighth  of a round sphere~$\frac12\sphere^3_{+++}$, respectively. 
Together with results in~\cite{Rad14,G-L3} this implies:

\begin{corollary}
The foliations given by the $\Spin(9)$ and $\Spin(9)\cdot\SO(2)$ 
actions on $\sphere^{31}$ are the only homogeneous foliations of 
spheres 
whose leaf space has constant curvature $4$
and which are not composed.
\end{corollary}


The case of composed foliations $(\sphere^{15},\fol_{0}\circ\fol_{C_{8,1}})$ 
is also very interesting, as they coincide with those foliations that 
contain the fibers of the octonionic Hopf fibration 
$\sphere^{15}\to \sphere^8$. Based on the fact that the Cayley projective 
plane $\mathbb O P^2$ is the mapping cone of $\sphere^{15}\to\sphere^8$, 
it was shown in~\cite{Rad14} that there corresponds to any 
singular Riemannian foliation $(\sphere^8,\fol_0)$ 
a singular Riemannian foliation $(\mathbb O P^2,\tilde\fol_0)$ 
which is homogeneous if and only if $\fol_0\circ\fol_{C_{8,1}}$ is homogeneous. 
It thus follows from our Main Theorem that there is a large amount of
inhomogeneous foliations of $\mathbb O P^2$:

\begin{corollary}
The foliation $(\mathbb O P^2,\tilde\fol_0)$ is inhomogeneous for
any foliation $(\sphere^8,\fol_0)$ except for those six (homogeneous) 
examples listed in Tables~\ref{Table 1} and~\ref{Table 4}.
\end{corollary}

About this paper: after a section on preliminaries, we first consider the case
of foliations with closed leaves and treat  
the cases $C_{9,1}$ and~$C_{8,1}$ in separate sections, 
as they have very different features.
The short, last section is devoted to foliations with non-closed
leaves. 

It is our pleasure to thank Alexander Lytchak for several very informative 
discussions as well as for his hospitality during our stay at the University 
of Cologne. 

\section{Preliminaries}

In this section,
we quickly review some definitions and results from~\cite{Rad14}.
\subsection{Singular Riemannian foliations} 
\begin{defn}
Let $M$ be a Riemannian manifold, and $\fol$ a partition of $M$ into complete, connected, injectively immersed submanifolds, called \emph{leaves}. The pair $(M,\fol)$ is called:
\begin{itemize}
\item A \emph{singular foliation} if there is a family of smooth vector fields $\{X_i\}$ that spans the tangent space of the leaves at each point.
\item A \emph{transnormal system} if any geodesic starting perpendicular to a leaf stays perpendicular to all the leaves it meets. Such geodesics are called \emph{horizontal geodesics}.
\item A \emph{singular Riemannian foliation} if it is both a singular foliation and a transnormal system.
\end{itemize}
\end{defn}

Given a singular foliation $(M,\fol)$, the  \emph{space of leaves}, denoted by $M/\fol$, is the set of leaves of $\fol$ endowed with the topology induced by the \emph{canonical projection} $\pi:M\to M/\fol$ that sends a point $p\in M$ to the leaf $L_p\in \fol$ containing it.
If in addition the leaves of $\fol$ are closed, then $M/\fol$ inherits the structure of a Hausdorff metric space, by declaring the distance $d(\pi(p),\pi(q))$ to be equal to the distance $d(L_p,L_q)$ in $M$ between the corresponding leaves. Moreover, $M/\fol$ is stratified by smooth Riemannian manifolds, and the projection $\pi$ is global submetry, and a Riemannian submersion along each stratum.

\subsection{Clifford systems and Clifford foliations}\label{clif-sys}
A Clifford system, denoted by $C$, is an $(m+1)$-tuple $C=(P_0,\ldots, P_m)$ of symmetric transformations of a Euclidean vector space 
$V$ such that
\[
P_i^2=Id\quad\mbox{for all $i$,}\quad P_iP_j=-P_jP_i\quad\mbox{for all
$i\neq j$.}
\]
Two Clifford systems $(P_0,\ldots, P_{m})$, $(Q_0,\ldots, Q_m)$ are called \emph{geometrically equivalent} if there exists an element $A\in \OO(V)$ such that $(AP_0A^{-1},\ldots, AP_mA^{-1})$ and $(Q_0,\ldots, Q_m)$ span the same subspace in $\Sym^2(V)$. Geometric equivalence classes of Clifford systems are completely classified, and:
\begin{itemize}
\item A Clifford system $\{P_0,\ldots, P_m\}$ on $V$ exists if and only if $\dim V=2k\delta(m)$, where $k$ is a positive integer and $\delta(m)$ is given by:
\begin{equation*}\label{E:table-delta-m}
\begin{array}{|c||c|c|c|c|c|c|c|c|c|}\hline \;m\; & \;1\; & \;2\; &\; 3\; & \;4\; & \;5\; &\; 6\; & \;7\; &\; 8\; &\; 8+n\; \\\hline \delta(m) & 1 & 2 & 4 & 4 & 8 & 8 & 8 & 8 & 16\delta(n)\\ \hline \end{array}
\end{equation*}
Given integers $m$, $k$, we denote by $C_{m,k}$ any Clifford system consisting of $m+1$ symmetric matrices on a vector space of dimension $2k\delta(m)$.
\item If $m\not\equiv 0\mod4$, there exists a unique geometric equivalence class of Clifford system of type $C_{m,k}$, for any fixed $k$.
\item If $m\equiv0\mod4$, there exist exactly $\left\lfloor{\frac k2}\right\rfloor+1$ equivalence classes of Clifford systems of type $C_{m,k}$. They are distinguished by the invariant $|\tr(P_0P_1\cdots P_m)|$.
\end{itemize}

Given a Clifford system $C=(P_0,\ldots, P_m)$ on $\RR^{2l}$, $l=k\delta(m)$, we can define a map
\begin{align*}
\pi_C:\sphere^{2l-1}\subset \RR^{2l}&\lra \RR^{m+1}\\
x&\lmt (\scal{P_0x,x},\ldots, \scal{P_mx,x}).
\end{align*}
The \emph{Clifford foliation} $(\sphere^{2l-1},\fol_C)$ associated to $C$ is given by the preimages of the map $\pi_C$.
This foliation is a singular Riemannian foliation, it only depends on the geometric equivalence class of $C$, and its quotient is isometric to either a round sphere $\frac12\sphere^m$ if $l=m$, or a round hemisphere $\frac12\sphere^{m+1}_+$ if $l\geq m+1$.

\subsection{Composed foliations} Fix a Clifford system 
$C=C_{m,k}=(P_0,\ldots,P_m)$ with associated Clifford foliation $(\sphere^{n},\fol_C)$, and fix a singular Riemannian foliation $(\sphere^m,\fol_0)$. 
Alternatively, we can view $\fol_0$ as: a 
foliation of the boundary of the leaf space of 
$\fol_C$, namely $\partial(\sphere^n/\fol_C)=\partial(\frac12\sphere^{m+1}_+)$,
in case $l\geq m+1$; and a foliation of $\frac12\sphere^m$ in 
case~$l=m$.
Such a foliation can be extended  by homotheties to a foliation 
$(\frac12\sphere^{m+1}_+,\fol_0^h)$. The \emph{composed foliation} 
$(\sphere^n,\fol_0\circ\fol_C)$ is then defined by taking the 
$\pi_C$-preimages of the leaves of $\fol_0^h$. 

Given any Clifford system $C=C_{m,k}$ and any singular Riemannian 
foliation $(\sphere^m,\fol_0)$, the composed foliation 
$(\sphere^{n},\fol_0\circ\fol_C)$ is a singular Riemannian foliation.

\subsection{Homogeneous composed foliations}
Recall that a singular Riemannian foliation $(M,\fol)$ is called \emph{homogeneous} if its leaves are orbits of an isometric Lie group action $G\to \Iso(M)$. In \cite{Rad14} appears a complete classification of homogeneous Clifford foliation and a partial classification of composed foliations:

\begin{thm}[\cite{Rad14}]\label{homog}
Let $C=C_{m,k}=(P_0,\ldots,P_m)$ be a Clifford system on $\RR^{2l}$ and let $(\sphere^m,\fol_0)$ be a singular Riemannian foliation. Then:
\begin{enumerate}
 \item The Clifford foliation $(\sphere^{2l-1},\fol_C)$ is homogeneous if and only if $m=1$, $2$ or $m=4$ and $P_0P_1\cdots P_4=\pm Id$, in which cases it is respectively spanned by the orbits of the diagonal action of $\SO(k)$ on $\RR^k\times \RR^k$ ($m=1$), $\SU(k)$ on $\CC^k\times \CC^k$ ($m=2$) or $\Sp(k)$ on $\HH^k\times \HH^k$ ($m$=4).
\item If $C\neq C_{8,1}$, $C_{9,1}$ then $(\sphere^{2l-1},\fol_0\circ \fol_C)$ is homogeneous if and only if both $\fol_0$ and $\fol_C$ are homogeneous. If $C=C_{9,1}$ and $(\sphere^{2l-1},\fol_0\circ \fol_C)$ is homogeneous, then $\fol_0$ is homogeneous.
\end{enumerate}
\end{thm}

By the classification of Clifford systems, both
$C_{8,1}$ and $C_{9,1}$ consist of a unique geometric equivalence class of Clifford systems. Moreover, for $C=C_{8,1}$ the corresponding Clifford foliation $(\sphere^{15},\fol_C)$ is given by the fibers of the octonionic Hopf fibration $\sphere^{15}\to\frac12\sphere^8$, while for $C=C_{9,1}$ the Clifford foliation $(\sphere^{31},\fol_C)$ is given by the fibers of $\pi_C:\sphere^{31}\to\frac12\sphere^{10}_+$.

\section{The case \texorpdfstring{$C=C_{9,1}$}{c91}}\label{c91}

In this section we will show that there are no new examples of 
homogeneous composed foliations originating from the Clifford system 
$C=C_{9,1}$. More precisely, we will see that a composed foliation
$(\sphere^{31},\fol_0\circ\fol_C)$ is homogeneous if and only if 
$\fol_0$ is the codimension one foliation of $\sphere^{10}_+$ consisting 
of concentric $9$-spheres; recall that in that case,
the composed foliation is the isoparametric 
foliation $\tilde\fol_C$ of FKM-type given by the orbits 
of the spin representation $\Spin(10)\to\SO(32)$~\cite{FKM}.
Recall also that the maximal
connected Lie
subgroup of $\SO(32)$ whose orbits coincide with the leaves
of $\tilde\fol_C$ is 
$\Spin(10)\cdot\U(1)=\Spin(10)\times_{\mathbb Z_4}\U(1)$~\cite{D,EH}.

In this section we will only consider closed Lie subgroups of $\SO(32)$,
which correspond to \emph{proper} isometric actions on $\sphere^{31}$,
and postpone the case of non-closed Lie subgroups to section~\ref{non-closed}.
So suppose the leaves of 
$\mathcal F_0\circ\mathcal F_C$ are orbits
of a closed connected Lie subgroup $G$ of $\SO(32)$. 
Since $\fol_0\circ\fol_C$ is contained in $\tilde\fol_C$, i.e.~the leaves 
of $\mathcal F_0\circ\mathcal F_C$ are contained in those of 
$\tilde\fol_C$, $G$ preserves each leaf of $\tilde\fol_C$. 
By the above maximality property,
$G\subset \Spin(10)\cdot\U(1)$.

\begin{lem}\label{lem}
The foliation $(\sphere^{31},\fol)$ induced by 
$G\subset \Spin(10)\cdot\U(1)$ is of the form $\fol_0\circ\fol_C$ 
if and only if $\fol_C$ is contained in $\fol$.
\end{lem}

\begin{proof}
The only if part is clear. Suppose now that the orbits of $G$ contain the 
leaves of $\fol_C$. Any element in $\Spin(10)\cdot \U(1)$ preserves the 
submanifold $M_+\subset \sphere^{31}$ defined as the preimage of the north pole 
of $\sphere^{31}/\fol_C=\frac12\sphere^{10}_+$, and therefore so does $G$. 
Since $G$ acts by isometries, the projection of any $G$-orbit to the 
quotient $\frac12\sphere^{10}_+$ is either entirely contained in the interior 
of $\frac12 \sphere^{10}_+$ or entirely contained in the boundary. It follows 
that for every leaf $L$ of $\fol$, the restriction $(L,\fol_C|_L)$ is a 
regular foliation, and its quotient $L/\fol_C\subset \frac12 \sphere^{10}_+$ is 
a submanifold. The partition $\{L/\fol_C\}_{L\in \fol}$ is easily seen to 
form a singular Riemannian foliation $\fol_0^h$ on $\frac12\sphere^{10}_+$ 
with the north pole as a $0$-dimensional leaf and, by the Homothetic 
Transformation Lemma (see e.g.~\cite[Lemma~1.1]{Rad12}), this foliation 
is determined by its restriction $\fol_0$ on the boundary 
$\frac12\sphere^{9}$. By definition of composed foliation, $\fol$ is 
of the form $\fol_0\circ\fol_C$.
\end{proof}

It follows from Lemma~\ref{lem} that we need only consider maximal 
connected closed subgroups of $\Spin(10)\cdot\U(1)$. 

The orbital geometry of the spin representation $\Spin(10)\to\SO(32)$
(or its extension to $\Spin(10)\cdot\U(1)$)
is well understood. The orbit space $\sphere^{31}/\Spin(10)$ is isometric
to an interval of length $\pi/4$, where the endpoints parametrize 
singular orbits $M_+$, $M_-$ of dimensions $21$ and~$24$ 
(cf.~\cite[p.~436]{HPT};
see also~\cite[pp.~8-9]{Bry} for a more elementary discussion). The orbit 
$M_+$ is particularly interesting, as it is also a leaf of 
$\fol_0\circ\fol_C$ for any homogeneous
foliation $\fol_0$ of $\frac12\sphere^{10}_+$, 
namely, the $\pi_C$-fiber over the origin 
of $\frac12\sphere^{10}_+$. As a homogeneous space,
$M_+\cong \Spin(10)/\SU(5)\cong \Spin(10)\cdot \U(1)/\U(5)$
(this also follows from the fact that $M_+$ is the orbit of a 
highest weight vector of the spin representation). 
Since $G$ is transitive on $M_+$, we must have $\dim G\geq21$. 

The maximal connected closed subgroups of $\Spin(10)\cdot \U(1)$ are,
up to conjugacy,
\[ \Spin(10),\ \U(5)\cdot \U(1),\ \Spin(10-k)\cdot \Spin(k)\cdot \U(1) \]
for $k=1,\ldots,5$, and
\[ \rho(H)\cdot \U(1) \]
where $H$ is simple and $\rho$ is irreducible of real type and 
degree~$10$ (cf.~\cite{Dyn1}; see also~\cite[Prop.~8]{K-P}). We have already remarked that 
$\Spin(10)$ is an orbit equivalent subgroup of $\Spin(10)\cdot\U(1)$; 
we shall not need to discuss its subgroups, because they are subgroups 
of the other maximal subgroups of $\Spin(10)\cdot\U(1)$.  
In the sequel, 
we first analyse which of the other maximal subgroups of $\Spin(10)\cdot\U(1)$
can act transitively on $M_+$.

The group $\U(5)\cdot \U(1)$ cannot act transitively on $M_+$ since its semisimple
part $\SU(5)$ is coincides with an isotropy subgroup of $\Spin(10)$ on $M_+$. 

The simply-connected compact connected simple Lie
groups $H$ of rank at most~$5$ and dimension between $20$
and $44$ are $\Spin(7)$, $\Spin(8)$, $\Spin(9)$, $\Sp(3)$, $\Sp(4)$, $\SU(5)$
and $\SU(6)$; none
admits irreducible representations of real type and degree~$10$. 

In order to determine if the groups $\Spin(10-k)\cdot \Spin(k)\cdot \U(1)$ can act
transitively on $M_+$, one can compute the intersection of the Lie algebra
$\mathfrak{so}(10-k)\oplus\mathfrak{so}(k)$ with the 
$\mathfrak{so}(10)$-isotropy subalgebra $\mathfrak{su}(5)$.
It does not matter that the subalgebras are defined only up to conjugacy
(corresponding to the fact that one can choose a different 
basepoint in~$M_+$).
We view $\mathfrak{su}(5)$ inside $\mathfrak{so}(10)$ as consisting
of matrices of the form
\[ \left(\begin{array}{cc}A&B\\-B&A\end{array}\right) \]
where $A$, $B$ are real $5\times 5$ matrices, $A$ is skew-symmetric,
$B$ is symmetric of trace zero. 
A standard choice of embedding of 
$\mathfrak{so}(10-k)\oplus\mathfrak{so}(k)$ into $\mathfrak{so}(10)$
is given by matrices of the form
\[ \left(\begin{array}{cc}C&0\\0&D\end{array}\right) \]
where $C$, $D$ are skew-symmetric 
$(10-k)\times(10-k)$, resp. $k\times k$, matrix blocks.
Then their intersection is isomorphic 
to~$\mathfrak{su}(5-k)\oplus\mathfrak{so}(k)$.
Therefore 
the dimension of the $\Spin(10-k)\cdot\Spin(k)$-orbit 
through the basepoint is~$21-\frac{k(k-1)}2$ for $k\leq4$, and $10$ for $k=5$.  
We deduce that $\Spin(9)$ and $\Spin(9)\cdot\U(1)$ 
act transitively on $M_+$; besides
those, 
only $\Spin(8)\cdot \SO(2)\cdot \U(1)$ has a chance of acting transitively
on that manifold. In order to discard the latter group, we choose a 
different embedding of $\mathfrak{so}(8)\oplus\mathfrak{so}(2)$
into $\mathfrak{so}(10)$, namely, that in which the $(i,j)$-entry
is zero if $i\in\{1,2,3,4,6,7,8,9\}$ and $j\in\{5,10\}$
or $j\in\{1,2,3,4,6,7,8,9\}$ and $i\in\{5,10\}$.
Now $(\mathfrak{so}(8)\oplus\mathfrak{so}(2))\cap\mathfrak{su}(5)\cong
\mathfrak{s}(\mathfrak{u}(4)\oplus\mathfrak{u}(1))$ and the 
corresponding $\Spin(8)\cdot \SO(2)$-orbit has dimension $29-16=13$,
showing that $\Spin(8)\cdot \SO(2)\cdot \U(1)$ is not transitive 
on~$M_+$.

Finally, 
we need to show that the $\Spin(9)\cdot \U(1)$-orbits  
cannot coincide with the leaves of $\mathcal F_0\circ\mathcal F_C$
for any $\mathcal F_0$. Suppose the contrary for some $\fol_0$. 
Since $\pi_C:\sphere^{31}\to \frac12\sphere^{10}_+$ 
is equivariant with respect to the double covering $\Spin(10)\to \SO(10)$, 
we see that~$\SO(9)$ preserves the leaves of~$\mathcal F_0$. 
We already know that~$\mathcal F_0$ is homogeneous (Theorem~\ref{homog}(2)), 
and $\SO(9)$ is a maximal connected 
subgroup of $\SO(10)$. Therefore $\fol_0$ must be given by the 
orbits of~$\SO(9)$. It follows that the leaf space of 
$\fol_0\circ\fol_C$ is $\frac12\sphere^{10}_+/\SO(9)$, which is 
isometric to $\frac12\sphere^2_{++}$. 
On the other hand, the quotient space 
$\sphere^{31}/\Spin(9)\cdot \U(1)$ is 
one-eighth of a round sphere 
$\frac12\sphere^2_{+++}$~\cite[Table~II, Type~$\mathrm{III_4}$]{St}.  
We reach a contradition and deduce that 
$(\sphere^{31},\mathcal F_0\circ\mathcal F_C)$ cannot be homogeneous
under $\Spin(9)\cdot\U(1)$.

\begin{rmk}
Let $(\sphere^{10}_+,\fol_0)$ denote the homogeneous foliation given by the orbits of 
$\SO(9)$, and let $(\sphere^{31},\fol_0\circ\fol_C)$ the corresponding composed 
foliation. By the result above, $\fol_0\circ\fol_C$ is not homogeneous and, 
in particular, it is different from the homogeneous foliation induced by the orbits 
of $\Spin(9)\cdot\U(1)$. Nevertheless, both foliations have cohomogeneity $2$, 
and both have quotients of constant curvature~4. Moreover, they both contain the 
homogeneous foliation induced by the orbits of $\Spin(9)$. 
Since $\sphere^{31}/\Spin(9)=\frac12\sphere^3_{++}$, 
the orbits of $\Spin(9)$ have codimension 1 in the leaves of $\fol_0\circ\fol_C$, 
which makes $\fol_0\circ\fol_C$ very close to a homogeneous foliation.
\end{rmk}

\section{The case \texorpdfstring{$C=C_{8,1}$}{c81}}\label{c81}

In this section, we determine the list of homogeneous composed foliations 
originating from the Clifford system $C=C_{8,1}$. Namely, we determine
the orbit equivalence classes of 
the isometric group actions that yield such foliations.
In this section we only consider closed subgroups of $\SO(16)$ and 
defer the analysis of non-closed Lie subgroups to section~\ref{non-closed}.
The foliation $\fol_C$ is given by the fibers of the 
inhomogeneous octonionic Hopf fibration 
$\sphere^{15}\to\frac12 \sphere^8$. Fix 
a singular Riemannian foliation $(\sphere^8,\fol_0)$, and suppose that $\fol_0\circ \fol_C$ is homogeneous, given by the orbits of a 
closed connected subgroup $G$ of~$\SO(16)$. 
Recall that if $X$ denotes the leaf space
\[
X=\sphere^8/\fol_0
\]
then the orbit space $\sphere^{15}/G$ is isometric to $\frac12X$. In particular, the sectional curvature of (the regular part of)
$\sphere^{15}/G$ is everywhere $\geq 4$ and hence $G$ cannot act polarly,
unless it acts with cohomogeneity~1.

\subsection{Criteria to recognize composed foliations}\label{SS:extending}
Before we start the classification in detail,  we want to present some results that will be helpful to identify foliations that can be written as $\fol_0\circ \fol_{C}$, where $C=C_{8,1}$. We start with the straightforward remark
that a foliation $\fol$ can be written in the form $\fol_0\circ \fol_C$ if and only if every fiber of the Hopf fibration
$\sphere^{15}\to\sphere^8$ is contained in a leaf 
of~$\fol$ (compare Lemma~\ref{lem}). 
In particular, if $\fol$ is a homogeneous composed foliation induced by the action of a group $G\subset \SO(16)$, then any other group $\ol{G}$ with $G\subset \ol{G}\subset \SO(16)$ will also generate a homogeneous composed foliation.

As a special case of the above situation, 
which will be useful later on, suppose that  
$(\sphere^{15},\fol_0\circ \fol_C)$ is homogeneous given by the orbits of 
$G\subset\SO(16)$, and suppose that $(\sphere^8,\fol_0)$ is homogeneous given by 
the orbits of $H\subset \SO(9)$. Then for any group $\ol{H}\subset \SO(9)$ containing
$H$, there is a canonical enlargement $\ol{G}\subset \SO(16)$ of $G$ whose orbits
yield a composed folation, as follows. Since the Hopf fibration 
$\sphere^{15}\to \sphere^8$ is equivariant with respect to the covering map 
$\Spin(9)\to \SO(9)$, we can lift $\ol{H}$ to a group 
$\tilde{H}\subset\Spin(9)\subset \SO(16)$. Now $\ol{G}$ is defined 
as the closure of the subgroup in $\SO(16)$ generated by $G$ and $\tilde{H}$. 
By the discussion above,
the orbits of $\ol{G}$ define a homogeneous 
composed foliation on $\sphere^{15}$.

Next we prove a criterion to distinguish some foliations that cannot be written in the form $\fol_0\circ \fol_C$.
\begin{prop}\label{P:noncomposed}
  Let $(\sphere^{15},\ol{\fol})$ denote the homogeneous, codimension 1 foliation given by the orbits of $\Sp(2)\cdot\Sp(2)$, under the representation 
$\nu_2{\hat\otimes_{\mathbb H}}\nu_2^*$. 
Then any foliation $(\sphere^{15},\fol)$ which is contained in $\ol{\fol}$ (i.e., every leaf of $\fol$ is contained in a leaf of $\ol{\fol}$) cannot be written in the form $\fol_{0}\circ \fol_{C_{8,1}}$.
\end{prop}

\begin{proof}
If $\fol$ could be written as $\fol_0\circ \fol_{C_{8,1}}$, 
by the remarks above, so could $\ol{\fol}$. Therefore it is enough to 
prove the proposition for $\ol{\fol}$ and, to do so, 
it is enough to provide a leaf of $\ol{\fol}$ that cannot be foliated 
by totally geodesic 7-spheres. We thus consider the singular orbit~$M_+$ 
containing the point $\mathrm{Id}\in\mathrm{Hom}_{\mathbb R}(\HH^2,\HH^2)\cong 
\HH^2\otimes_{\HH} \HH^{2*}$, which is diffeomorphic to $\Sp(2)$. 

Suppose now that $M_+=\Sp(2)$ is foliated by totally geodesic $\sphere^7$. Then the leaves are all simply connected, which implies that there is no leaf holonomy, and thus the quotient $M_+/\fol$ is a manifold $B$ and $M_+\to B$ is a Riemannian submersion with totally geodesic fibers. Then it is also a fibration, and from the long exact sequence in homotopy, $B$ is simply connected (and 3-dimensional). Therefore it must be $B=\sphere^3$, and we have a fibration $\sphere^7\to \Sp(2) \to \sphere^3$. Again from the long exact sequence in homotopy, we have
\[
\pi_6(\Sp(2))\lra \pi_6(\sphere^3) \lra \pi_5(\sphere^7)=0
\]
However, on the one hand $\pi_6(\Sp(2))=0$ (for example, cf.~\cite{MT}), 
and on the other $\pi_6(\sphere^3)\neq0$, which gives a contradiction.
\end{proof}

As an application of Proposition \ref{P:noncomposed} above, consider the Clifford foliation $\fol_{\bar{C}}$ generated by $\bar{C}=(P_0,\ldots, P_4)$ with $P_0P_1P_2P_3P_4=\pm Id$. This foliation is homogeneous and given by the orbits of the diagonal action of $\Sp(2)$ on $\HH^2\oplus \HH^2$ 
(Theorem~\ref{homog}) and thus, by Proposition \ref{P:noncomposed} above, it cannot be written as $\fol_0\circ \fol_{C_{8,1}}$.  In fact, this is the only Clifford foliation of $\sphere^{15}$ with this property:

\begin{prop}\label{P:onlyCbar}
For any Clifford system $C'$ on $\RR^{16}$ with $C'\neq \bar{C}$, 
the foliation $\fol_{C'}$  can be written in the form 
$\fol_{C'}=\fol_0\circ\fol_{C_{8,1}}$, for some foliation $(\sphere^8,\fol_0)$.
\end{prop}
\begin{proof}
Let $C_{8,1}=(P_0,\ldots, P_8)$ and, for every $i=1,\ldots,7$, let 
$C_i$ denote the sub-Clifford system $(P_0,\ldots, P_i)$. Since 
$\fol_{C_i}$ is given by the preimages of the map
\begin{align*}
\pi_{C_i}:\sphere^{15}&\to \RR^{i+1},\\
\pi_{C_i}(x)&=\big(\scal{P_0x,x},\ldots, \scal{P_ix,x}\big),
\end{align*}
it is clear that $\pi_{C_i}$ factors as $\pi_i\circ\pi_{C_{8,1}}$, 
where $\pi_{C_{8,1}}:\sphere^{15}\to \sphere^8$ is the Hopf fibration, 
and $\pi_i:\sphere^8\subset\RR^9\to \DD^{i+1}\subset\RR^{i+1}$ is the projection 
onto the first $i+1$ components. In particular, $\fol_{C_i}$ can be written as 
$\fol_0\circ \fol_{C_{8,1}}$, where $(\sphere^8,\fol_0)$ is given by the fibers 
of $\pi_i$. Notice that $\fol_0$ in this case is homogeneous and given by the 
orbits of $\SO(8-i)$, 
embedded in $\SO(9)$ as the lower diagonal block.

Moreover, any Clifford system $C'=C_{m,k}$ on $\RR^{16}$ must satisfy the equation $k\delta(m)=8$, and the only possibilities are
\[
(m,k)= (8,1), (7,1), (6,1), (5,1), (4,2), (3,2), (2,4), (1,8).
\]

For any $m\not\equiv 0\mod 4$ there is only one geometric equivalence class of Clifford systems, and therefore $C_{m,k}$ can be identified with the sub-Clifford system $C_m\subset C_{8,1}$. For $m\equiv 0\mod 4$ there are exactly 
$\lfloor\frac k2\rfloor+1$ geometrically distinct Clifford systems of type $C_{m,k}$. Therefore, there is a unique $C_{8,1}$, and two distinct classes of type $C_{4,2}$. One of them is $C_4\subset C_{8,1}$, which is composed by the discussion above, and the other is $\bar{C}$. Since this exhausts all possible Clifford systems on $\RR^{16}$, it follows that all of them are composed, with the exception of $\bar{C}$.
\end{proof}

Gathering all the information together, we obtain the following
\begin{cor}\label{C:composed}
A composed foliation $(\sphere^{15}, \fol_0\circ \fol_{C_{m,k}})$ can also be written as $\fol_0'\circ\fol_{C_{8,1}}$ for some $(\sphere^8,\fol_0')$, if and only if $C_{m,k}\neq \bar{C}$.
\end{cor}
\begin{proof}
If $C_{m,k}\neq \bar{C}$ then, by Proposition \ref{P:onlyCbar}, $\fol_{C_{m,k}}$ can be written as $\fol_0'\circ\fol_{C_{8,1}}$ and, by the initial remark, the same holds for $\fol_0\circ\fol_{C_{m,k}}$ since it contains $\fol_{C_{m,k}}$. On the other hand, any composed foliation $\fol_0\circ \fol_{\bar{C}}$ is contained in the foliation $\fol_1\circ \fol_{\bar{C}}$ where $(\sphere^8,\fol_1)$ is the trivial foliation with one leaf. Since $\fol_1\circ\fol_{\bar{C}}$ coincides with the foliation $\ol{\fol}$ of Proposition \ref{P:noncomposed}, 
$\fol_0\circ\fol_C$ cannot be written as $\fol_0'\circ \fol_{C_{8,1}}$ for any $(\sphere^8,\fol_0')$.
\end{proof}
\medskip

We can now proceed with the classification of composed foliations 
of $\sphere^{15}$ homogeneous under a closed Lie group~$G$. 
The diameter of $X={\sphere^8/\fol_0}$ is either equal to~$\pi$, or 
it is at most~$\pi/2$. 
We will consider these two cases separately.

\subsection{Case I: $\diam X=\pi$} Suppose first that the diameter of $X$ is 
$\pi$. Then there is a copy of $\sphere^0\subset \sphere^8$ consisting of 
$0$-dimensional leaves, and $(\sphere^8,\fol_0)$ decomposes as a spherical join
\[
(\sphere^{8},\fol_0)=\sphere^0\star (\sphere^7,\fol_1),
\]
for some foliation $\fol_1$. In particular, $X$ is isometric to a spherical 
join $\sphere^0\star Y$, where $Y=\sphere^7/\fol_1$. In this case, $\frac12X$ has diameter $\pi/2$ (thus $G$ acts reducibly) and it contains two points 
$x_+$, $x_-$
at distance~$\pi/2$. Moreover, any unit speed geodesic in 
$\frac12X$ starting from $x_-$ meets $x_+$ at the same time $t=\pi/2$. 
Therefore the preimages $S_{\pm}$ of $x_{\pm}$ are orthogonal round spheres 
of curvature 1, i.e.,~they are the unit spheres of subspaces $V_{\pm}$ of 
$\RR^{16}$ such that $\RR^{16}=V_+\oplus V_-$. Since we are assuming that the 
$G$-orbits contain the fibers of the Hopf fibration, it must be 
$\dim S_{\pm}\geq 7$. Therefore equality must hold, and $\dim V_+=\dim V_-=8$. 
Moreover, $G$ acts transitively on $S_{\pm}$. Given $p\in S_+$, the isotropy 
$G_p$ acts on the unit sphere in the normal space $\nu_pS_+$, which is
isometric to $S_-$ via the 
map $v\mapsto \exp_p{\frac\pi2}v$.
Moreover, the foliation $(S_-, G_p)$ 
coincides with the infinitesimal foliation of 
$\fol_0$ at $\pi_C(p)\in \sphere^8$, 
which in turn coincides with $(\sphere^7,\fol_1)$. In particular, $\fol_0$ is 
homogeneous and given by the action of  $G_p$ on $\RR^9=\RR\oplus V_-$ given by 
$\epsilon\oplus \lambda|_{G_p}$, where $\epsilon:G\to \RR$ is the trivial 
representation, and $\lambda:G\to \SO(8)$ denotes the representation of 
$G$ on $V_-$ (or $V_+$).
\begin{rmk}\label{R:symmetric}
Since the infinitesimal foliation of $\fol_0\circ \fol_C$ at any point of $S_-$ coincides with the infinitesimal foliation at a point in $S_+$ (because they both coincide with $(\sphere^7,\fol_1)$), the slice representations at $S_+$ and $S_-$ must be orbit equivalent.
\end{rmk}
If the $G$ action on $S_{\pm}$ is not effective, then the kernels $K_{\pm}$ of $G\to \SO(V_{\pm})$ are normal subgroups of $G$ with $K_+\cap K_-=\{e\}$. 
Since $G$ is compact, it admits a normal subgroup~$L$ such that 
$G=K_+\cdot L\cdot K_-$, where $K_+\cdot L$ acts effectively on $S_-$ and $L\cdot K_-$ acts effectively on $S_+$. Let $\kk_+,\kk_-,\Ll$ denote the Lie algebras of $K_+$, $K_-$, $L$ respectively. From the list of all groups acting transitively on the $7$-sphere, we get the following possibilities:

\smallskip

\paragraph{1.~$\kk_+=\kk_-=0$} Then $G=L$ up to a finite cover, and the possible such representations are:

\begin{center}
\begin{tabular}{|c|c|c|}
\hline
Type&$G$ &$G\to \SO(16)$\\ \hline
I.1&$\SO(8)$&$\rho_8\oplus\rho_8$\\
I.2&$\SU(4)$&$\mu_4\oplus\mu_4$ \\
I.3&$\U(4)$&$\mu_4\oplus\mu_4$\\ \hline
II.1&$\Spin(8)$&$\Delta_8^+\oplus \Delta_8^-$\\
II.2&$\Spin(7)$&$\Delta_7\oplus \Delta_7$\\
II.3&$\SU(4)\cdot\U(1)$ & $\mu_4\hat\otimes(\mu_1^r\oplus\mu_1^s)$ ($r\neq s$) \\ \hline
III.1&$\Sp(2)$&$\nu_2\oplus\nu_2$\\
III.2&$\Sp(2)\cdot\Sp(1)$&$(\nu_2\hat{\otimes}\nu_1)^{\oplus 2}$\\
III.3&$\Sp(2)\cdot\U(1)$&$\nu_2\hat{\otimes}(\mu_1^r\oplus \mu_1^s)$\\ \hline
\end{tabular}
\end{center}

\medskip

The actions of type I induce the Clifford foliations $\fol_{C_{1,8}}$ and 
$\fol_{C_{2,4}}$ respectively (actions~I.2 and~I.3 are orbit equivalent) and, by 
Proposition~\ref{P:onlyCbar}, they indeed can be written as 
$\fol_0\circ \fol_{C_{8,1}}$. Therefore, the same is true for the 
foliations coming from actions of type~II, since each of them contains 
a foliation of type~I. On the other hand, the foliations of 
type~III are containted in the orbits of the representation
of $\Sp(2)\cdot\Sp(2)$ given by $\nu_2\hat\otimes_{\mathbb H}\nu_2^*$, 
and therefore are not of the form 
$\fol_0\circ \fol_{C_{8,1}}$ by Proposition \ref{P:noncomposed}. Therefore, 
the homogeneus composed foliations in this case are given by 
the orbits of the groups listed in Table~\ref{Table 1}, 
where we have put together orbit equivalent actions. 
As we have seen, the foliation $(\sphere^{8},\fol_0)$ is also homogeneous, 
given by the orbits of the isotropy group $H$ of $G$ 
at a certain point.

{\small

\begin{table}[!htb]
\begin{center}
\begin{tabular}{|c|c||c|c||c|}
\hline
$G$ for $\fol_0\circ\fol_C$&$G\to \SO(16)$&$H$ for $\fol_0$&$H\to \SO(9)$& $X$\\ \hline\hline
$\Spin(8)$&$\Delta_8^+\oplus \Delta_8^-$&$\Spin(7)$&$\epsilon\oplus\Delta_7$&$[0,\pi]$\\ \hline
$\SO(8)$&$\rho_8\oplus\rho_8$&$\SO(7)$&$\epsilon^2\oplus \rho_7$& $\sphere^2_+$ \\
$\Spin(7)$&$\Delta_7\oplus \Delta_7$&$\mathrm{G}_2$&$\epsilon^2\oplus\phi_7$&\\ \hline
$\SU(4)$&$\mu_4\oplus\mu_4$&$\SU(3)$&$\epsilon^3\oplus \mu_3$& $\sphere^3_+$ \\
$\U(4)$&$\mu_4\oplus\mu_4$&$\U(3)$&$\epsilon^3\oplus \mu_3$ &\\ \hline
$\SU(4)\cdot\U(1)$ & $\mu_4\hat\otimes(\mu_1^r\oplus\mu_1^s)$ ($r\neq s$)&$\U(3)\cdot\U(1)$&$\epsilon\oplus \mu_1^{r-s}\oplus \mu_3\otimes\mu_1^{-s}$& $\sphere^2_{++}$ \\ \hline
\end{tabular}
\end{center}
\caption{ $\diam X=\pi$, and $\kk_+=\kk_-=0$.}
\label{Table 1}
\end{table}

}

\begin{rmk}
(a) Any pair of equivalent or inequivalent $8$-dimensional
irreducible representations of $\Spin(8)$ could occur in the table, but 
some are not listed since they differ from the two listed by an outer 
automorphism of $\Spin(8)$. 
In particular those representations are not only orbit equivalent to a representation in the list, but their image in $\SO(16)$ is the same as the image of a
representation in the list. 




 \end{rmk}


\smallskip

\paragraph{2.~$\Ll=0$.} Then $G=K_+\cdot K_-$, 
and each $K_{\pm}$ acts transitively on $\sphere^7$. All these cases are orbit equivalent among themselves, and also to the first entry in~Table~1, so we get no 
new examples.



\smallskip

\paragraph{3.~$\Ll\neq0$ and $\kk_+\neq0$.} Since $L\cdot K_+$ is a nontrivial product, and it acts effectively and transitively on $S_-$, it must be 
\[ L\cdot K_+\in \{\Sp(2)\cdot\Sp(1),\SU(4)\cdot\U(1), \Sp(2)\cdot\U(1)\}. \]

If $L=\Sp(2)$, then 
the foliation 
is contained in the foliation of Proposition~\ref{P:noncomposed},
so it is not composed.

If $L=\SU(4)$ then $K_+=\U(1)$, and $K_-$ can be either 
$\U(1)$ or trivial.
Then $G$ is given by $\U(1)\cdot\SU(4)\cdot\U(1)$,
resp., $\U(1)\cdot\SU(4)$, and it acts via 
$(\mu_1\hat{\otimes}\mu_4)\oplus (\mu_4\hat{\otimes}\mu_1)$, 
resp., 
$(\mu_1\hat{\otimes}\mu_4)\oplus \mu_4$. 
Those actions are orbit equivalent to the representation of $\SU(4)\cdot\U(1)$ in Table \ref{Table 1} given by $\mu_4\hat{\otimes}(\mu_1^r\oplus\mu_1^s)$ with $r\neq s$ (including the case $(r,s)=(1,0)$). 

Finally, if $L=\Sp(1)$ or $\U(1)$, then $K_+$, $K_-\in \{\Sp(2),\SU(4)\}$ and the action has cohomogeneity 1, and they  
are all orbit equivalent to the first entry in Table~\ref{Table 1}.

Hence we get no new examples in this case.

\subsection{Case II: $\diam X\leq\pi/2$}In this case the diameter of $\sphere^{15}/G=\frac12X$ is at most $\pi/4$ and thus $G$ acts irreducibly. We 
distinguish between possible cases, according to the dimension of $X$.


Suppose first that $\dim X=1$, i.e., $\fol_0\circ \fol_C$ is an 
isoparametric family in $\sphere^{15}$. It follows from the 
classification of cohomogenity~1 actions in spheres that the only 
possible actions on $\sphere^{15}$ with quotient of diameter $\leq\pi/4$ 
are given by $\nu_2\hat\otimes\nu_2^*$ for $G_1=\Sp(2)\cdot\Sp(2)$, 
$\mu_2\hat\otimes\mu_4$ for $G_2=\mathrm{S}(\U(2)\cdot\U(4))$, and  
$\rho_2\hat\otimes\rho_8$ for $G_3=\SO(2)\cdot\SO(8)$ 
(or $\SU(2)\cdot\SU(4)$ and  $\SO(2)\cdot\Spin(7)$, 
which are orbit equivalent subgroups of $G_2$, $G_3$, resp.). 
By Proposition~\ref{P:noncomposed}, the action of $G_1$ is ruled out,
but the other two actions  
give rise to composed foliations; in fact, those actions yield 
foliations containing foliations given in
Table~1. Since $G_2$ and $G_3$ are contained in $\Spin(9)$, in each case
they project to a subgroup $H$ of $\SO(9)$ which generates a codimension one
isoparametric foliation $\fol_0$ in $\sphere^8$.
We summarize the discussion above in the following table:

\begin{table}[!htb]
\begin{center}
\begin{tabular}{|c|c||c|c||c|}
\hline
$G$ for $\fol_0\circ\fol_C$&$G\to \SO(16)$&$H$ for $\fol_0$&$H\to \SO(9)$& $X$\\ \hline\hline
$\SU(2)\cdot\SU(4)$ &$\mu_2\hat{\otimes}_{\CC}\mu_4$& $\SO(3)\cdot\SO(6)$&$\rho_3\hat\oplus \rho_6$&$[0,\pi/2]$\\ \hline
$\SO(2)\cdot\SO(8)$ &$\rho_2\hat{\otimes}_{\RR}\rho_8$& $\SO(2)\cdot\SO(7)$&$\rho_2\hat\oplus \rho_7$&$[0,\pi/2]$\\ \hline
\end{tabular}\end{center}
\caption{$\diam X=\pi/2$.}
\label{Table 4}
\end{table}

If $2\leq\dim X\leq 4$, then $G$ acts irreducibly on $\sphere^{15}$ with 
cohomogeneity $\leq 4$, and the action is not polar. 
From the classification of low cohomogeneity 
representations in~\cite{H-L,St,G-L}, it follows that $G$ must act on 
$\sphere^{15}$ with cohomogeneity 2, and 
there are exactly two
 possible actions, $\mu_2\hat\otimes_{\mathbb C}\nu_2$
for $G_1=\U(2)\cdot\Sp(2)$, and $S^3(\mu_1)\hat\otimes_{\mathbb H}\nu_2^*$ for
$G_2=\SU(2)\cdot\Sp(2)$~\cite[Table~II]{St}. 
Again these actions are ruled out by Proposition~\ref{P:noncomposed}.
In fact it is clear that $G_2$ is contained in $\Sp(2)\cdot\Sp(2)$. 
As for $G_1$ being contained in that group,
note that the $\Sp(2)$-representation $\CC^4$ restricts 
to $\CC^2\oplus\CC^{2*}$ along the embedding $\U(2)\subset\Sp(2)$,
so the result follows from the following 
representation theoretic lemma. 
 
\begin{lem}
If $V$ and $W$ are representations of complex, resp., quaternionic type, 
then $(V\oplus V^*)\otimes_{\mathbb H}W^*$ is equivalent as a real representation
to the realification of $V\otimes_{\mathbb C}W$. 
\end{lem}

\begin{proof}
The representations have equivalent complexifications.
Indeed the complexification of the first representation is 
$(V\oplus V^*)\otimes_{\mathbb C}W$ whereas that of the second is 
$(V\otimes_{\mathbb C} W)\oplus(V\otimes_{\mathbb C} W)^*$, where $W\cong W^*$
over $\CC$.  
\end{proof}

If $\dim X\geq 5$, then the foliation $(\sphere^8,\fol_0)$ has leaves of 
dimension $\leq 3$ and, by \cite{Rad12}, it is homogeneous. We claim that 
there are no composed homogeneous foliations in this case.

First of all, the regular 
leaves of $\fol_0$ cannot have dimension~$1$ o~$2$ 
(i.e., $\dim X\neq 6$, $7$). In fact, in those cases $\fol_0$ would have to 
be generated by a representation $H\subset \SO(9)$, where $H=S^1$ or $T^2$. 
In particular, $H$ would be contained in a maximal torus of $\SO(9)$, and 
every such maximal torus acts on $\sphere^8$ fixing at least two antipodal 
points. In particular the diameter of $X$ would be $\pi$ which contradicts 
our assumption.

We are thus left with the case in which $(\sphere^8,\fol_0)$ is homogeneous
under a closed connected subgroup $H$ of $\SO(9)$ 
and $\dim X=5$. For the same reasons above, $\fol_0$ cannot be generated by 
a $T^3$-action. The principal orbits are $3$-dimensional with effective 
(transitive) actions of $H$. Therefore a principal isotropy group 
$H_{princ}$ does not contains a normal subgroup of $H$, $H_{princ}$ is a 
subgroup of $\OO(3)$, $\dim H\leq6$ and equality holds if and only if 
$H$ is locally isomorphic to $\SU(2)\times\SU(2)$. 
We deduce that $H$ is one of $\SU(2)$, $\SU(2)\times T^1$,  
$\SU(2)\times \SU(2)$, up to cover.

The only almost effective $9$-dimensional representation of $\SU(2)\times\SU(2)$ without fixed directions is $\rho_3\hat{\otimes}\rho_3$, 
which has $6$-dimensional principal orbits. 

Assume $H=\SU(2)\times T^1$ and $V$ is a $9$-dimensional
representation with cohomogeneity~$6$ and no fixed directions. 
The identity component of $H_{princ}$ on $V$ is a circle with non-trivial
projection into $\SU(2)$. It follows that  
the only admissible irreducible
components of $V$ are $(\SU(2),\RR^3)$, $(\U(2),\CC^2)$, $(T^1,\CC)$.
Since $9$ is odd, the first representation must occur exactly once. 
We get two possiblities: $\RR^3\oplus\CC^2\oplus\CC$ and $\RR^3\oplus\CC\oplus\CC\oplus\CC$. 
The first one has trivial principal isotropy groups, so it is excluded. 
The second one can be extended to an action of $\ol H=\SU(2)\times T^3$
acting on $\sphere^8$ with cohomogeneity~$3$. If $\fol_0\circ \fol_C$ were homogeneous, induced by some group $G$, then the extension $\ol{H}$ of $H$ would induce an extension $\ol{G}$ of $G$ that would act on $\sphere^{15}$ with cohomogeneity~$3$. This action would be non-polar
and irreducible, however there is no such group~\cite[Table~1]{G-L}.


The only 9-dimensional representations of $H=\SU(2)$ without 
fixed directions are $\lambda_9$, $\mu_2\oplus \lambda_5$ and $\rho_3\oplus\rho_3\oplus \rho_3$.

The representation $\rho_3\oplus\rho_3\oplus \rho_3$ can be extended to 
an action of $\ol{H}=\SO(3)^3$ via the outer sum 
$\rho_3\hat{\oplus}\rho_3\hat{\oplus} \rho_3$, acting on $\sphere^8$ with 
cohomogeneity $2$. If $\fol_0\circ \fol_C$ were homogeneous, induced by 
some group $G$, then the extension $\ol{H}$ of $H$ would induce an 
extension $\ol{G}$ of $G$, that would act on $\sphere^{15}$, with quotient 
isometric to $\frac12\sphere^2_{+++}$ and 
three most singular orbits of dimension $9$ (they would be preimages of 
most singular $\ol{H}$ orbits, of dimension $2$). However, from the 
classification of non-polar 
irreducible isometric actions of cohomogeneity 2 on $\sphere^{15}$
there is no such group~\cite{St}, 
and therefore $\fol_0\circ \fol_C$ cannot be homogeneous in this case.

The representation $\mu_2\oplus \lambda_5$ can be extended to an action of $\ol{H}=\SU(2)\times\SU(2)$ via the representation $\mu_2\hat{\oplus}\lambda_5$, again acting on $\sphere^8$ with cohomogeneity 2. If $\fol_0\circ \fol_C$ were homogeneous, induced by some group $G$, then the extension $\ol{H}$ of $H$ would induce an extension $\ol G$ of $G$, with quotient isometric to $\frac12\sphere^8/\ol{H}=\frac12(\sphere^2_+/D_3)$, where $D_3$ denotes a dihedral group. 
The group $\ol{G}$ must then be $\Sp(1)\cdot\Sp(2)$~\cite{St}, which is $13$-dimensional and thus acts on $\sphere^{15}$ with finite principal isotropy. In particular $G$ must act with finite principal isotropy as well, and since the cohomogeneity of $H$ on~$\sphere^8$ is~$5$, we have $\dim G=10$. However, a quick check shows that there are no~$10$-dimensional groups of rank at most~$3$
acting irreducibly (and non-polarly) 
on~$\RR^{16}$. In particular in this case $\fol_0\circ\fol_C$ cannot be homogeneous.

The representation $\lambda_9$ has isolated singular orbits,  and therefore the quotient $X$ has no boundary (compare~\cite[\S~11.2]{G-L}). Now suppose that the composed foliation $\fol_0\circ\fol_C$ is homogeneous, given by the action of $G$ on $\sphere^{15}$. Since the quotient $\frac12X$ has no boundary, there are no \emph{nontrivial reductions} of $(G,\RR^{16})$, i.e., there are no other representations $(G',\RR^n)$ with $\dim G'<\dim G$ such that $\sphere^{n-1}/G'$ is isometric to $\sphere^{15}/G=\frac12X$, cf.~\cite[Prop.~5.2]{G-L}. In particular, $G$ must act with trivial principal isotropy, since otherwise we could produce a nontrivial reduction~\cite[p.~2]{G-L}. Since the principal isotropy is trivial and $\dim X=5$, again it must be $\dim G=10$.
The only~10-dimensional group acting irreducibly (and non-polarly)
on~$\RR^{16}$ is $G=\SU(2)^3\times\U(1)$ acting by
$\hat\otimes^3(\mu_2)\hat\otimes\mu_1$; however a pure tensor 
$v_1\otimes v_2\otimes v_3$ has isotropy subgroup $T^3$, so this action
has as an orbit of dimension~$7$. Since $\lambda_9$ has no fixed 
points in $\sphere^8$, this shows that the $G$-orbits cannot yield a 
foliation of the form~$\fol_0\circ\fol_C$.

\section{Non-proper actions}\label{non-closed} 

We treat the cases of~$C=C_{9,1}$ and~$C=C_{8,1}$ simultaneously. 
Suppose $\fol_0\circ\fol_C$ is a homogeneous composed foliation 
of $\sphere^{31}$, resp., $\sphere^{15}$ 
given by the orbits of a non-closed connected Lie subgroup $G$ of $\SO(32)$, 
resp., $\SO(16)$. Then the closure of $G$ is a closed connected 
subgroup whose orbits also comprise a homogeneous composed foliation,
so it is already described in sections~\ref{c91} or~\ref{c81}. However,
most of the groups therein listed admit no dense non-closed 
connected Lie subgroups in view of the following:

\begin{lem}\label{dense}
A compact connected Lie group $U$ with
at most a one-dimensional center admits no dense non-closed 
connected Lie subgroups.
\end{lem}

\begin{proof}
Suppose, to the contrary, that $G$ is a dense connected proper 
Lie subgroup of~$U$. 
If $G$ is a normal subgroup of~$U$, then either $G$ is contained 
in the semisimple part of $U$ or it contains the center of $U$.
Owing to~\cite{Ra}, normal subgroups of semisimple Lie groups are closed. 
It follows that $G$ cannot be normal in~$U$.
Let $N$ be the normalizer of $G$
in $U$. This is a proper
subgroup of $U$, thus cannot be closed by denseness of $G$.
On the other hand, $N$ must be closed in $U$
because it coincides with the normalizer in~$U$
of the Lie algebra of~$G$
(here we use connectedness of $G$), a contradiction. 
\end{proof}

\medskip

The closed groups $U$ yielding homogeneous composed foliations 
described in sections~\ref{c91} or~\ref{c81} which do not satisfy
the assumptions of Lemma~\ref{dense} occur in case $C=C_{8,1}$ only 
and have two dimensional tori as
centers, and they are of two types:

\smallskip

\paragraph{1.} $U=K_+\cdot K_-$ where
\[ K_{\pm}\in\{\SU(4)\cdot\U(1), \Sp(2)\cdot\U(1)\}. \]
In both cases there are dense connected Lie subgroups $G$ which however 
yield orbit equivalent subactions. 

\paragraph{2.} $U=K_+\cdot L\cdot K_-$ where $K_{\pm}=\U(1)$,
$L=\SU(4)$, and $K_+\cdot L$ acts effectively on $\CC^4\oplus0$ and 
$L\cdot K_-$ acts effectively on $0\oplus\CC^4$. The non-closed dense connected 
Lie subgroups of $U$ are of the form $G=\RR\times\SU(4)$, where $\RR$
is an irrational line in the center $T^2$ of $U$. Note that $G$ and $U$ 
share a common singular orbit through $p\in\CC^4\oplus0$. Moreover
the isotropy groups at~$p$ act with the same orbits in $0\oplus\CC^4$. 
It follows that $G$ and $U$ are orbit equivalent on $\CC^4\oplus\CC^4$. 

Finally, we get no homogenous composed foliations with non-closed 
leaves.


\providecommand{\bysame}{\leavevmode\hbox to3em{\hrulefill}\thinspace}
\providecommand{\MR}{\relax\ifhmode\unskip\space\fi MR }
\providecommand{\MRhref}[2]{%
  \href{http://www.ams.org/mathscinet-getitem?mr=#1}{#2}
}
\providecommand{\href}[2]{#2}

\end{document}